\documentclass[12pt]{amsart}
\usepackage{fullpage}
\usepackage{latexsym}
\usepackage{amssymb}
\usepackage[all]{xy}

\newtheorem{thm}{Theorem}[section]
\newtheorem{prop}[thm]{Proposition}
\newtheorem{cor}[thm]{Corollary}
\newtheorem{con}[thm]{Conjecture}
\newtheorem{lem}[thm]{Lemma}
\newtheorem{que}[thm]{Question}

\theoremstyle{remark}
\newtheorem{defn}[thm]{Definition}
\newtheorem{rem}[thm]{Remark}

\usepackage{hyperref}

\newcommand{\CP}{\mathbb{CP}}

\newcommand{\Q}{\mathbb{Q}}
\newcommand{\R}{\mathbb{R}}
\newcommand{\Z}{\mathbb{Z}}

\newcommand{\SO}{\mathrm{\SO}}

\newcommand{\Nil}{\mathrm{Nil}}

\newcommand\rank{\operatorname{rank}}
\newcommand\im{\operatorname{im}}

\title[{\tiny Geometric structures, the Gromov order, Kodaira dimensions $\&$ simplicial volume}]{Geometric structures, the Gromov order, Kodaira dimensions and simplicial volume}

\author{Christoforos Neofytidis and Weiyi Zhang}
\address{Department of Mathematics, Ohio State University, Columbus, OH, USA}
\email{neofytidis.1@osu.edu}
\address{Mathematics Institute, University of Warwick, Coventry, CV4 7AL, UK}
\email{Weiyi.Zhang@warwick.ac.uk}
\date{\today}

\begin{document}

\maketitle

\begin{abstract}
We introduce an axiomatic definition for the Kodaira dimension and classify Thurston geometries in dimensions $\leq 5$ in terms of this Kodaira dimension. We show that the Kodaira dimension is monotone with respect to the partial order defined by maps of non-zero degree between 5-manifolds. We study the compatibility of our definition with traditional notions of Kodaira dimension, especially the highest possible Kodaira dimension. To this end, we establish a connection between the simplicial volume and the holomorphic Kodaira dimension, which in particular implies that any smooth K\"ahler $3$-fold with non-vanishing simplicial volume has top holomorphic Kodaira dimension.
\end{abstract}

\section{Introduction}

The Kodaira dimension provides a very successful classification tool  for complex manifolds. This concept has been generalised by several authors to symplectic manifolds, especially in dimensions two and four~\cite{LeBrun1,LeBrun2,Li1,DS1,DS2, DZ}, to almost complex manifolds~\cite{CZ},  as well as to manifolds with a geometric decomposition in the sense of Thurston in dimensions three and four~\cite{Zhang, Li2}. Our first goal in the present article is to generalise the traditional notions of Kodaira dimensions by introducing a more systematic study of the Kodaira dimension $\kappa^g$ for manifolds that carry a geometric structure, especially in the sense of Thurston, and provide a complete classification in dimensions $\leq 5$. 

Our proposed approach takes into account both coarse geometric (e.g., curvature) and group theoretic (e.g., fundamental group) structures of the manifold. On the one hand, manifolds that contain factors with compact universal coverings are assigned the lowest possible Kodaira dimension ($-\infty$). On the other hand, the presence of some form of hyperbolicity on the manifold, as generalised by the notion of irreducible locally symmetric spaces of non-compact type (or the non-vanishing of the simplicial volume as we shall explain below), motivates the highest possible value (half of the dimension). These values generalise the well-known case of holomorphic Kodaira dimension for surfaces. Beyond the two ends, we treat in a uniform way solvable (Euclidean-by-Euclidean) geometries and introduce a more systematic consideration of half values as implied by the existence of fiber bundle structures for spaces that do not fall into the two boundary values of the Kodaira dimension.

A significant question in topology, suggested by Gromov~\cite{Gromovbook} and Milnor-Thurston~\cite{MT}, is whether a given numerical homotopy invariant $\iota\in[0,\infty]$ is monotone with respect to maps of non-zero degree, that is, whether the existence of a map of non-zero degree $M\longrightarrow N$ implies $\iota(M)\geq\iota(N)$. In~\cite{Neoorder,Zhang} this question was answered in the affirmative for the Kodaira dimension of manifolds of dimension $\leq 3$ and for geometric manifolds in dimension four. In this paper we show the monotonicity of the Kodaira dimension of geometric 5-manifolds.

\begin{thm}
Let $M$ and $N$ be two closed oriented geometric 5-manifolds. If there is a map of non-zero degree from $M$ to $N$, then $\kappa^g(M)\geq\kappa^g(N)$. 
\end{thm}

This clearly implies the following.

\begin{cor}
Let $M$ and $N$ be two closed oriented geometric 5-manifolds. If there are maps of non-zero degree $M\rightleftarrows N$, then $\kappa^g(M)=\kappa^g(N)$. 
\end{cor}

One of the most prominent examples of monotone invariants is the Gromov norm~\cite{Gromov}. For a topological space $X$ and a homology class $\alpha\in H_n(X;\R)$, the Gromov norm of $\alpha$ is defined to be
\[
\|\alpha\|_1:=\inf \biggl\{\sum_{j} |\lambda_j| \ \biggl\vert  \ \sum_j \lambda_j\sigma_j\in C_n(X;\R) \ \text{is a singular cycle representing } \alpha  \biggl\}.
\]
If $X$ is a closed oriented $n$-dimensional manifold, then the Gromov norm or simplicial volume of $X$ is given by $\|X\|:=\|[X]\|_1$, where $[X]$ denotes the fundamental class of $X$. 
The simplicial volume satisfies an even stronger condition than monotonicity: If $f\colon M\longrightarrow N$ is a map of degree $\deg(f)$, then 
\begin{equation}\label{eq.functorial}
\|M\|\geq|\deg(f)|\|N\|,
\end{equation}
and equality holds when $f$ is a covering map. The non-vanishing of the simplicial volume is a powerful tool to show non-existence of maps of non-zero degree, and the classification of Kodaira dimension suggests that manifolds with top Kodaira dimension are those with non-vanishing simplicial volume. Results of Gromov for hyperbolic manifolds~\cite{Gromov} and Lafont-Schmidt~\cite{LS} and Bucher~\cite{Bucher} for irreducible locally symmetric spaces of non-compact type will motivate one of our building axioms, namely to set the Kodaira dimension of a closed manifold $M$ in the above classes to be $$\kappa^g(M)=\frac{\dim M}{2}.$$
Together with vanishing results, our choice indeed establishes the following connection between top Kodaira dimension and simplicial volume.

\begin{thm}\label{t:topKodGromov}
A closed geometric 5-manifold $M$ has non-zero simplicial volume if and only if $\kappa^g(M)=\frac{5}{2}$.
\end{thm}

It is natural to examine the compatibility of our Kodaira dimension with the existing notions of Kodaira dimensions, such as the holomorphic Kodaira dimension $\kappa^h$ for complex $2n$-manifolds and the symplectic Kodaira dimension $\kappa^s$ of minimal symplectic 4-manifolds. As Theorem \ref{t:topKodGromov} suggests, for the top Kodaira dimension the positivity of the simplicial volume is the connecting principle. We thus need to answer the following questions for the holomorphic and symplectic Kodaira dimension respectively (in this paper we will concentrate on $\kappa^h$).

\begin{que}{\normalfont(\cite[Question 3.13]{Zhang}).}\label{topKod0norm} \
\begin{enumerate}
\item Let $M$ be a smooth $2n$-dimensional complex manifold with non-vanishing simplicial volume. Is $\kappa^h(M)=n$?
\item Let $M$ be a smooth $4$-dimensional symplectic manifold with non-vanishing simplicial volume. Is $\kappa^s(M)=2$?
\end{enumerate}
\end{que}

When $M$ is a K\"ahler surface, the above question was positively answered by Paternain and Petean~\cite{PP}, who showed that $M$ admits an $\mathcal F$-structure in the sense of Cheeger and Gromov~\cite{CG} if and only if the Kodaira dimension is different from two. 
The existence of an $\mathcal F$-structure implies the vanishing of the simplicial volume~\cite{CG,PP}. Moreover, all known examples of compact complex surfaces which are not of K\"ahler type have $\mathcal F$-structure and thus vanishing simplicial volume. In other words, the complex part of Question \ref{topKod0norm} for complex surfaces is reduced to answering the following: {\em Does every complex surface of Class VII have vanishing simplicial volume?} 

Here, we will address the first part of Question \ref{topKod0norm}, giving a uniform treatment in all dimensions, and 
an affirmative answer for K\"ahler $3$-folds will follow from results in algebraic geometry. 

\begin{thm}\label{t:posKod} \
\begin{enumerate}
\item If $M$ is a smooth complex projective $n$-fold with non-vanishing simplicial volume, then $\kappa^h(M)\neq n-1$, $n-2$ or $n-3$.
\item If $M$ is a smooth K\"ahler $3$-fold with non-vanishing simplicial volume, then $\kappa^h(M)=3$.
\end{enumerate}
\end{thm}

In fact, our argument shows that the first part of Question \ref{topKod0norm} for projective manifolds follows from two well known conjectures in algebraic geometry, due to Mumford and Koll\'ar (Conjectures  \ref{uniruled} and \ref{kod0ame} respectively). When the complex dimension is no greater than three, both conjectures are known to be true. The second part of Theorem \ref{t:posKod} then follows from algebraic approximations of compact K\"ahler $3$-folds. Moreover, the first part of Theorem \ref{t:posKod} is actually true for Moishezon manifolds and the second part works for complex $3$-folds of Fujiki class $\mathcal C$.

The authors are very grateful to Michelle Bucher and Tian-Jun Li for very useful discussions. Part of this work was carried out during collaborative visits of the first author at the University of Warwick and the second author at the University of Geneva. The authors would like to thank these institutions for their hospitality and stimulating environments. We also thank the anonymous referees for their useful feedback.

\section{The Kodaira dimension for Thurston geometries}

In this section we give a definition of the Kodaira dimension and classify in terms of this notion closed manifolds that possess a Thurston geometry in dimensions $\leq 5$.

Let $\mathbb{X}^n$ be a complete simply connected $n$-dimensional Riemannian manifold. We say that a closed manifold {\em $M$ is an $\mathbb{X}^n$-manifold}, or that {\em $M$ is modeled on $\mathbb{X}^n$}, or that {\em $M$ possesses the $\mathbb{X}^n$ geometry} in the sense of Thurston, if it is diffeomorphic to a quotient of
$\mathbb{X}^n$ by a lattice $\Gamma$ in the group of isometries of $\mathbb{X}^n$ (acting effectively and transitively). The group $\Gamma$ is the
fundamental group of $M$. Two geometries $\mathbb{X}^n$ and $\mathbb{Y}^n$ are the same whenever there exists a diffeomorphism $\psi \colon \mathbb{X}^n
\longrightarrow \mathbb{Y}^n$ and an isomorphism $\mathrm{Isom}(\mathbb{X}^n) \longrightarrow \mathrm{Isom}(\mathbb{Y}^n)$ which sends each $g \in \mathrm{Isom}(\mathbb{X}^n)$ to $\psi \circ g \circ \psi^{-1} \in \mathrm{Isom}(\mathbb{Y}^n)$.

\subsection{Axiomatic definition}\label{ss:ad}

Let $\mathcal G$ be the smallest class of manifolds that contains all
\begin{itemize}
\item points;
\item manifolds modeled on a compact geometry;
\item solvable manifolds;
\item irreducible symmetric manifolds of non-compact type;
\item fiber bundles or manifolds modeled on fibered geometries, whose fiber and base (geometries) belong in $\mathcal G$.
\end{itemize}
We define the {\em Kodaira dimension} $\kappa^g$ of an $n$-manifold $M\in\mathcal G$ as follows:
\begin{itemize}
\item[(A0)] If $M$ is a point, then $\kappa^g(M)=0$;
\item[(A1)] If $M$ is modeled on a compact geometry, then $\kappa^g(M)=-\infty$;
\item[(A2)] If $M$ is solvable;
then $\kappa^g(M)=0$;
\item[(A3)] If $M$ is irreducible symmetric of non-compact type, then $\kappa^g(M)=\frac{n}{2}$;
\item[(A4)] If $M$ is a fiber bundle or is modeled on a fibered geometry $\mathbb F\to\mathbb X^n\to\mathbb B$, and does not satisfy any of (A1)-(A3), then
$$
\kappa^g(M)=\sup_{F,B}\{\kappa^g(F)+\kappa^g(B)\},
$$
where the supremum runs over all possible manifolds $F$ and $B$ that occur in a fibration $F\to M\to B$ or are modeled on $\mathbb F$ and $\mathbb B$ respectively, and which  satisfy one of the Axioms (A1)-(A3),
\end{itemize}

An immediate consequence of the above definition is the following.

\begin{lem}\label{l:cover}
Let  $M\in\mathcal G$ and suppose $\overline M\to M$ is a finite covering. Then $\overline M\in\mathcal G$ and $\kappa^g(\overline M)=\kappa^g(M)$.
 \end{lem}
 
\subsection{Classification in dimensions $\leq5$}

\subsection*{Dimension zero}
The Kodaira dimension of a point is equal to zero by (A0).

\subsection*{Dimension one}
The only closed 1-manifold is the circle $S^1=\R/\Z$, i.e., it is modeled on the real line. In particular, $S^1$ is solvable satisfying (A2), hence
\[
\kappa^g(S^1)=0.
\]

\subsection*{Dimension two}
Let $\Sigma_h$ be a surface of genus $h$. If $h=0$, then $\Sigma_0=S^2$ satisfies (A1). If $h=1$, then $\Sigma_1=T^2=\R^2/\Z^2$, i.e., it possesses the Euclidean geometry $\R^2$ which satisfies (A2). Finally, if $h\geq2$, then $\Sigma_h$ is hyperbolic, that is, it is modeled on $\mathbb{H}^2$ and satisfies (A3). Hence, to summarise, we have
\begin{equation*}\label{eq.dim2}
\kappa^g(\Sigma_h)= \left\{\begin{array}{ll}
        -\infty, & \text{if } h=0;\\
        \ \ 0, & \text{if } h=1;\\
        \ \ 1, & \text{if } h\geq2.
        \end{array}\right.
\end{equation*}

\subsection*{Dimension three}
By Thurston's geometrization picture in dimension three, there exist eight geometries~\cite{Thu,Scott}. The compact geometry $S^3$ satisfies (A1), the geometries $\R^3$ (Euclidean), $Nil^3$ (nilpotent) and $Sol^3$ (solvable but not nilpotent) satisfy (A2), and the hyperbolic geometry satisfies (A3). We are left with the three product geometries which do not belong to (A1)-(A3). For the geometry $S^2\times\R$ we have, according to (A4) and the Kodaira dimensions for 1- and 2-manifolds,
\[
\kappa^g(S^2\times S^1)=\kappa^g(S^2)+\kappa^g(S^1)=-\infty.
\]
Every 3-manifold $M$ modeled on $\mathbb{H}^2\times\R$ or $\widetilde{SL_2}$ is a finitely covered by a circle bundle over a closed hyperbolic surface $\Sigma_h$. Hence, (A4), Lemma \ref{l:cover} and the Kodaira dimensions for the circle and hyperbolic surfaces, give us
\[
\kappa^g(M)=\kappa^g(S^1)+\kappa^g(\Sigma_h)=1.
\]
Summarising,
\begin{equation*}\label{eq.dim3}
\kappa^g(M)= \left\{\begin{array}{ll}
        -\infty, & \text{if } M \text{ is modeled on } S^3, \text{ or } S^2\times\R;  \\
        \ \ 0, & \text{if } M \text{ is modeled on } \R^3, Nil^3 \text{ or }  Sol^3;\\
        \ \ 1, & \text{if } M \text{ is modeled on } \mathbb{H}^2\times\R \text{ or } \widetilde{SL_2};\\
        \ \ \frac{3}{2},  & \text{if } M \text{ is modeled on } \mathbb{H}^3.
        \end{array}\right.
\end{equation*}

\subsection*{Dimension four}
In his 1983 thesis, Filipkiewicz~\cite{Filipkiewicz} classified the 4-dimensional geometries. According to that, there exist nineteen geometries, eighteen of which have representatives which are compact manifolds. We now enumerate those geometries following our Axiomatic Definition \ref{ss:ad}. For the notation and details on the structure of each geometry and of manifolds modeled on them, we refer to Filipkiewicz's thesis~\cite{Filipkiewicz}, as well as to papers of Wall~\cite{Wall1,Wall2} and Hillman's monograph~\cite{Hillman}; see also~\cite{NeoIIPP} for some new characterizations for certain geometries of nilpotent and solvable type.

There are three compact geometries, namely $S^4$, $\CP^2$ and $S^2\times S^2$, and those satisfy (A1). Thus, a manifold $M$ modeled on any of those geometries has Kodaira dimension
\[
\kappa^g(M)=-\infty.
\]

There are six solvable geometries satisfying (A2): The Euclidean $\R^4$, the nilpotent $Nil^4$ and $Nil^3\times\R$, and the three solvable but not nilpotent geometries $Sol_0^4$, $Sol_1^4$ and $Sol_{m,n}^4$ (note that $Sol_{m,m}^4=Sol^3\times\R$). Hence for those geometries we have
\[
\kappa^g(M)=0.
\]

Next, (A3) is satisfied by the real and complex hyperbolic geometries, $\mathbb{H}^4$ and $\mathbb{H}^2(\mathbb C)$ respectively, as well as by the irreducible $\mathbb{H}^2\times\mathbb{H}^2$ geometry. We thus have for a manifold $M$ that possesses one of those geometries
\[
\kappa^g(M)=\frac{4}{2}=2.
\]

Finally, we deal with the remaining seven geometries which satisfy (A4): If a manifold $M$ is modeled on one of the geometries $S^2\times\R^2$, $S^2\times\mathbb{H}^2$ or $S^3\times\R$, then it has a finite cover which is fiber bundle with $S^2$- or $S^3$-fiber. Thus $\kappa^g(M)=-\infty$, because $\kappa^g(S^n)=-\infty$ for $n\geq2$. A manifold $M$ modeled on one of the geometries $\mathbb{H}^2\times\R^2$ or $\widetilde{SL_2}\times\R$ has Kodaira dimension $\kappa^g(M)=1$ by the corresponding classifications in lower dimensions, and, for the same reason, if $M$ is an $\mathbb{H}^3\times\R$-manifold, then $\kappa^g(M)=\frac{3}{2}$. Finally, if $M$ is modeled on the reducible geometry $\mathbb{H}^2\times\mathbb{H}^2$, then it is virtually a product of two hyperbolic surfaces, hence $\kappa^g(M)=2$ by the fact that hyperbolic surfaces have Kodaira dimension one. Note that $\kappa^g(M)=2=\frac{4}{2}$ for irreducible $\mathbb{H}^2\times\mathbb{H}^2$-manifolds, as we have seen above.

All this is summarised as follows
\begin{equation*}\label{eq.dim4}
\kappa^g(M)= \left\{\begin{array}{ll}
        -\infty, & \text{if } M \text{ is modeled on } S^4, \CP^2, S^2\times \mathbb{X}^2 \text{ or } S^3\times\R;\\
        \ \ 0, & \text{if } M \text{ is modeled on } \R^4, Nil^4, Nil^3\times\R, Sol_{m,n}^4, Sol_0^4 \text{ or } Sol_1^4;\\
        \ \ 1, & \text{if } M \text{ is modeled on } \mathbb{H}^2\times\R^2 \text{ or } \widetilde{SL_2}\times\R;\\
        \ \ \frac{3}{2},  & \text{if } M \text{ is modeled on } \mathbb{H}^3\times\R;\\
         \ \ 2,  & \text{if } M \text{ is modeled on } \mathbb{H}^4,  \mathbb{H}^2(\mathbb C) \text{ or } \mathbb{H}^2\times\mathbb{H}^2.
        \end{array}\right.
\end{equation*}

\subsection*{Dimension five}
Recently, Geng~\cite{Geng1} gave a classification of the 5-dimensional geometries. According to Geng's list, there exist fifty eight geometries, and fifty four of them are realised by compact manifolds. (Counting from Geng's list one finds fifty nine geometries, because the geometry $Sol^3\times\R^2$, which is $Sol_{m,n}^4\times\R$ for $m=n$, is counted individually.) As before, we will enumerate those geometries following Axioms (A0)-(A4). For a detailed description of each geometry, as well as for the terminology, we refer to the three papers from Geng's thesis~\cite{Geng1,Geng2,Geng3} and to related work; see the references in~\cite{Geng1}. In particular, for the virtual properties of a manifold modeled on each geometry, we refer to the individual sections/results as given in the statements of~\cite[Theorem 1.1]{Geng2} and~\cite[Theorem 1.1]{Geng3}. These descriptions will be used as well in Section \ref{s:monotonicity}. Furthermore, as it is remarked in~\cite[Section 4]{Geng1}, a similar classification for the Thurston geometries was partially done in dimensions six and seven (and thus the Kodaira dimensions of those manifolds can be similarly determined).

\subsubsection*{Manifolds satisfying (A1).}

There are three compact geometries: the 5-sphere $S^5$, the Wu symmetric manifold $SU(3)/SO(3)$ and the product $S^2\times S^3$. A manifold $M$ modeled on these geometries has Kodaira dimension
\[
\kappa^g(M)=-\infty.
\]

\subsubsection*{Manifolds satisfying (A2).}

Naturally, this is one of the most rich classes of new geometries with the various (irreducible) extensions of solvable-by-solvable geometries. There are two nilpotent and six solvable but not nilpotent extensions of type $\R^4\rtimes\R$, denoted by
\[
A_{5,1}, \ A_{5,2} \text{ and } A_{5,7}^{a,b,-1-a-b}, \ A_{5,7}^{1,-1-a,-1+a}, \ A_{5,7}^{1,-1,-1}, \ A_{5,8}^{-1}, \ A_{5,9}^{-1,-1}, \ A_{5,15}^{-1}
\] 
respectively. There are two nilpotent geometries of type $Nil^4\rtimes\R$, denoted by $A_{5,5}$ and $A_{5,6}$. There is one nilpotent and one solvable but not nilpotent geometry of type $(\R\times Nil^3)\rtimes\R$ denoted by $A_{5,3}$ and $A_{5,20}^0$ respectively. There is a solvable but not nilpotent extension $\R^3\rtimes\R^2$ denoted by $A_{5,33}^{-1,-1}$. The last irreducible solvable geometry is $Nil^5$. The remaining solvable geometries are built out of products of lower dimensional geometries: the Euclidean $\R^5$, the nilpotent $Nil^3\times\R^2$, $Nil^4\times\R$, and the solvable but not nilpotent $Sol_0^4\times\R$, $Sol_1^4 \times\R$, $Sol_{m,n}^4\times\R$ (note that $Sol_{m,m}^4\times\R=Sol^3\times\R^2$). A manifold $M$ modeled on any of these geometries has Kodaira dimension
\[
\kappa^g(M)=0.
\]

\subsubsection*{Manifolds satisfying (A3).}

Any manifold modeled on one of the irreducible symmetric geometries of non-compact type $\mathbb{H}^5$ or $SL(3,\R)/SO(3)$ has Kodaira dimension
\[
\kappa^g(M)=\frac{5}{2}.
\]

\subsubsection*{Manifolds satisfying (A4).}

A manifold $M$ modeled on any of the following geometries 
\begin{equation*}
\begin{array}{llll}
         S^2\times S^2\times\R & S^2\times\R^3 & S^2\times Nil^3 & S^2\times Sol^3\\
         S^2\times\mathbb H^2\times\R & S^2\times\widetilde{SL_2} &  S^2\times\mathbb H^3, &S^2\times\mathbb H^3\\
         S^3\times\R^2 & S^3\times\mathbb H^2, & S^4\times\R &  \CP^2\times\R\\
          Nil^3\times_\R S^3 & \widetilde{SL_2}\times_\alpha S^3 & L(a,1)\times_{S^1}L(b,1) & T^1(\mathbb H^3)
        \end{array}
\end{equation*}
satisfies (A4) with fiber or base one of the compact geometries $S^2$, $S^3$, $S^4$ or $\CP^2$. Hence $\kappa^g(M)=-\infty$ by the classification of Kodaira dimensions of manifolds of dimension $\leq4$.

Now, a manifold $M$ modeled on one of the geometries 
\begin{equation*}
\begin{array}{lll}
         \R^3\times\mathbb H^2 & Nil^3\times\mathbb H^2 & Sol^3\times\mathbb H^2\\
        \widetilde{SL_2}\times\R^2 & \R^2\rtimes\widetilde{SL_2} & Nil^3\times_\R\widetilde{SL_2}
\end{array}
\end{equation*}
is fibered with involved geometries $\mathbb H^2$ and a solvable geometry. Hence $\kappa^g(M)=1$.

Every representative $M$ of the $\mathbb H^3\times\R^2$ geometry satisfies (A4), where the supremum is achieved with the geometries $\mathbb H^3$ and $\R^2$, i.e., $\kappa^g(M)=\frac{3}{2}$.

Next, we deal with 5-manifolds which are fibrations over a space of Kodaira dimension two, namely they are modeled on one of geometries
\begin{equation*}
\begin{array}{lll}
         \mathbb H^2\times\widetilde{SL_2} & \mathbb H^2\times\mathbb H^2\times\R & \widetilde{SL_2}\times_\alpha\widetilde{SL_2} \\
         \mathbb H^4\times\R & \mathbb H^2(\mathbb C)\times\R & \widetilde{U(2,1)/U(2)}.
\end{array}
\end{equation*}
Indeed, those geometries are fibered over one of the geometries $\mathbb H^2\times\mathbb H^2$, $\mathbb H^2$ or $\mathbb H^2(\mathbb C)$. Hence, any 5-manifold modeled on the above geometries has Kodaira dimension $\kappa^g(M)=2$.

Finally, a manifold $M$ modeled on the product geometry $\mathbb H^2\times\mathbb H^3$ has top Kodaira dimension $\kappa^g(M)=1+\frac{3}{2}=\frac{5}{2}$.

We summarise the Kodaira dimensions of geometric 5-manifolds below.
\begin{equation*}\label{eq.dim5}
\kappa^g(M)= \left\{\begin{array}{lll}
        -\infty, & \text{if } M \text{ is modeled on } SU(3)/SO(3), S^5, S^2\times \mathbb{X}^3, S^3\times \mathbb{X}^2, S^4\times\R, \CP^2\times\R,  \\
                   & \text{ } \hspace{3.5cm} Nil^3\times_\R S^3, \widetilde{SL_2}\times_\alpha S^3, L(a,1)\times_{S^1}L(b,1) \text{ or } T^1(\mathbb H^3);\\
        \ \ 0, & \text{if } M \text{ is modeled on }  \R^5, \R^4\rtimes\R, \R^3\rtimes\R^2, Nil^5, Nil^4\rtimes\R, (\R\times Nil^3)\rtimes\R,\\ 
        & \text \ \hspace{3.5cm} Nil^4\times\R, Nil^3\times\R^2, Sol_0^4\times\R, Sol_1^4\times\R \text{ or }\\
        & \text \ \hspace{3.5cm}  Sol_{m,n}^4\times\R;\\
        \ \ 1, & \text{if } M \text{ is modeled on }  \mathbb{H}^2\times\R^3, \mathbb{H}^2\times\Nil^3, \mathbb{H}^2\times Sol^3, \R^2\times\widetilde{SL_2},  \R^2\rtimes\widetilde{SL_2}\\
         & \text \ \hspace{3.5cm} \text{or } Nil^3\times_\R\widetilde{SL_2};\\
        \ \ \frac{3}{2},  & \text{if } M \text{ is modeled on }  \mathbb{H}^3\times\R^2;\\
         \ \ 2,  & \text{if } M \text{ is modeled on }  \mathbb{H}^2\times\widetilde{SL_2}, \widetilde{SL_2}\times_\alpha\widetilde{SL_2}, \mathbb{H}^2\times \mathbb{H}^2\times\R, \mathbb{H}^4\times\R, \\  
         & \text \ \hspace{3.5cm} \mathbb{H}^2(\mathbb C)\times\R \text{ or } \widetilde{U(2,1)/U(2)};\\
         \ \ \frac{5}{2},  & \text{if } M \text{ is modeled on }   \mathbb{H}^5, SL(3,\R)/SO(3) \text{ or } \mathbb H^3\times\mathbb H^2.
        \end{array}\right.
\end{equation*}

\subsection{Remarks on the definition and classification} 

\subsubsection{Half integers and bundle additivity}
Half integers for the Kodaira dimension were introduced in~\cite{Zhang} for hyperbolic 3-manifolds. This is a natural development, taking into account the known top Kodaira dimension for complex manifolds and the simplicial volume; see also Section \ref{s:Gromovnorm}. Moreover, an additivity condition for fiber bundles was introduced in~\cite{LZ}, similarly to Axiom (A4). 
 Hence, although in~\cite{Zhang} the Kodaira dimension for $\mathbb H^3\times\R$ is defined to be one, it seems natural to define it to be equal to $\frac{3}{2}$. Indeed, a closed 4-manifold $M$ modeled on $\mathbb H^3\times\R$ is finitely covered by a product $F\times S^1$, where $F$ is a hyperbolic 3-manifold. Since solvable manifolds (in this case, the circle) have Kodaira dimension zero (by (A2)), we obtain the value
\[
\kappa^g(M)=\kappa^g(F)+\kappa^g(S^1)=\frac{3}{2}.
\]
The requirement on the supremum in (A4) becomes now clear: If $F$ is a mapping torus of a pseudo-Anosov diffeomorphism of a hyperbolic surface $\Sigma$ (every hyperbolic 3-manifold is virtually of this form~\cite{Agol}), then $M$ is a fiber bundle $\Sigma\to M\to T^2$. In that case, $\Sigma$ is irreducible locally symmetric of non-compact type, the 2-torus is solvable and therefore $\kappa^g(\Sigma)+\kappa^g(T^2)=1$. The supremum however is achieved with the fibration $F\to M\to S^1$.

Note that the monotonicity result for the Kodaira dimension of 4-manifolds with respect to maps of non-zero degree given in~\cite[Theorem 1.2]{Neoorder} is not affected with this new value for $\mathbb H^3\times\R$-manifolds. In fact, it reveals exactly the difference with the two 4-dimensional geometries with Kodaira dimension one, namely $\mathbb{H}^2\times\R^2
$ and $\widetilde{SL_2}\times\R$: As shown in~\cite[Theorem 1.1]{Neoorder}, not only no $\mathbb H^3\times\R$-manifold admits a map of non-zero degree from a manifold modeled on one of the geometries $\mathbb{H}^2\times\R^2$ or $\widetilde{SL_2}\times\R$, but, moreover, given any manifold $N$ which is modeled on one of the latter two geometries, then there is an $\mathbb H^3\times\R$-manifold $M$ and a map $M\longrightarrow N$ of non-zero degree. 

\subsubsection{Generalised Class VII surfaces}

Our Kodaira dimension is compatible with the holomorphic one for K\"ahler manifolds. However, according to Axiom (A2), the Kodaira dimension for $Sol_0^4$- and $Sol_1^4$-manifolds is zero instead of $-\infty$ as defined in~\cite{Zhang}. This is  again compatible with Axiom (A4), because those geometries are solvable-by-solvable, and lower dimensional solvable geometries have Kodaira dimension zero. In~\cite{Zhang}, the Kodaira dimension for $Sol_0^4$- and $Sol_1^4$-manifolds was defined to be $-\infty$ following Wall's scheme for complex non-K\"ahler surfaces~\cite{Wall2}. 
We could have required the Kodaira dimension of those manifolds, as well as of $Sol_{m\neq n}^4$-manifolds, to be indeed $-\infty$ as they have vanishing virtual second Betti number and thus admit no symplectic structures. 
However, in this paper, we have chosen to introduce the Kodaira dimension taking a unified value (zero) for solvable manifolds, keeping thus our axiomatic approach natural with the least possible assumptions.

\begin{rem}
Axiom (A4) is also strongly related to the Iitaka conjecture~\cite{Ii}, which states that the holomorphic Kodaira dimension for an algebraic fibration $F\to M\to B$ satisfies
$$
\kappa^h(M)\geq\kappa^h(F) + \kappa^h(B).
$$
In fact, our set of Axioms matches with the picture of Iitaka fibration in algebraic geometry, which is applied to compute the simplicial volume in Section \ref{s:Gromovnorm}.
\end{rem}

\subsubsection{Geometries with no compact representatives}

A phenomenon that appears in dimensions four and above is that of geometries that have no compact representatives, but still manifolds with finite volume. Being interested mostly in the monotonicity of the Kodaira dimension with respect to non-zero degree maps, and thus in compact manifolds, we omitted those geometries from our classification. It is nevertheless worth giving their values:
\begin{itemize}
\item In dimension four, the geometry $\mathbb F^4$ is realised by $T^2$-bundles over punctured hyperbolic surfaces~\cite{Hillman}. According to our axioms, any manifold $M$ modeled on $\mathbb F$ has $$\kappa^g(M)=1.$$ This coincides with the definition given in~\cite{Zhang}.
\item In dimension five, one has the geometries $\mathbb F^4\times\R$, $T^1(\R^{1,2})=\R^3\rtimes SO(1,2)^0/SO(2)$, and two quotients of $Nil^3\rtimes\widetilde{SL_2}$, which are denoted by $\mathbb F_0^5$ and $\mathbb F_1^5$. A manifold $M$ modeled on any of those geometries has virtually the structure of a circle bundle over an $\mathbb F^4$-manifold~\cite{Geng3}, hence it has Kodaira dimension $$\kappa^g(M)=1.$$
\end{itemize}

\subsubsection{Beyond Thurston's geometries}

The definition and classification of Kodaira dimension goes well beyond Thurston's geometries. Such a classification was given in~\cite{Zhang} for 3-manifolds, following the torus and sphere decompositions for 3-manifolds. One cannot hope for such a general result in higher dimensions based on geometric structures, as there exist manifolds that possess no geometric structures or decompositions. Moreover, there are diffeomorphic K\"ahler $n$-folds with different Kodaira dimensions when $n\ge 3$ \cite{Ras}.
Nevertheless, following decomposition results in dimension four~\cite{Hillman} and developing a similar theory for the recently classified geometries in dimension five one should be able to associate a numerical homotopy invariant for a much wider class of manifolds that will contain Thurston's geometries which is monotone with respect to maps of non-zero degree. We might include more manifolds by considering decomposition with pieces of Einstein manifolds.

Furthermore, our definition includes many more general classes that are not geometric. For example, in dimension four, the Kodaira dimension of a (not necessarily geometric) fiber bundle $F\to M\to B$ is
\begin{equation*}\label{eq.surfacebundles}
\kappa^g(M)= \left\{\begin{array}{ll}
        -\infty, & \text{if one of } F,B  \text{ is } S^2  \text{ or finitely covered by } \#_{m\geq0}S^2\times S^1;\\
        \ \ 0, & \text{if } F=B=T^2, \text{ or one of } F,B \text{ is a 3-manifold which is not finitely}\\
        & \text{covered by }\#_{m\geq0}S^2\times S^1  \text{ and contains no } \mathbb H^2\times\R,\widetilde{SL_2}  \text{ or } \mathbb H^3  \text{ pieces in its}\\
        &  \text{torus or sphere decomposition};\\
         \ \ 1, & \text{if one of } F,B  \text{ is } T^2  \text{ and the other is hyperbolic, or one of } F,B  \text{ is a}\\
         & \text{3-manifold which has at least one } \mathbb H^2\times\R \text{ or } \widetilde{SL_2} \text{ piece and no } \mathbb H^3  \text{ pieces}\\
         &  \text{in its torus or sphere decomposition};\\
          \ \ \frac{3}{2}, & \text{if one of } B,F  \text{ is a 3-manifold with at least one } \mathbb H^3  \text{ piece in its torus or}\\
        &  \text{sphere decomposition};\\
        \ \ 2, & \text{if both } F \text{ and } B \text{ are hyperbolic surfaces}.
        \end{array}\right.
\end{equation*}

The connection to the simplicial volume suggested by Theorem \ref{t:topKodGromov} is apparent: For the above fibration, $\|M\|>0$ if and only if $F$ and $B$ are hyperbolic surfaces~\cite[Corollary 1.3]{BN}. 
Also, this definition should be absolutely compatible with maps of non-zero degree. Namely, Gromov asks whether, given any manifold $N$, we can find a surface bundle $M$ and a map $M\longrightarrow N$ of non-zero degree~\cite[pg. 753, {\em Topological version of Bogomolov’s question}]{Gromovmaps}. 

\section{The Gromov order}\label{s:monotonicity}

Given two closed oriented $n$-manifolds $M$ and $N$, we say that $M$ {\em dominates} $N$ if there is a map $M\longrightarrow N$ of non-zero degree, and we denote this by $M\geq N$. In 1978, Gromov suggested studying the domination relation as a partial order~\cite{CT}. In dimension two, the domination relation is a total order given by the genus, as it can be easily seen that $\Sigma_g\geq\Sigma_h$ if and only if $g\geq h$. In higher dimensions, however, such an order is impossible. Nevertheless, various results have been obtained with respect to this order by many authors~\cite{Bel,BBM,CT,KN,Neoorder,Rong,Wang1}. As suggested by the monotonicity of the simplicial volume (inequality (\ref{eq.functorial})), one hopes to be able to understand whether a numerical invariant is monotone with respect to the domination relation; see Gromov~\cite{Gromovbook} and Milnor-Thurston~\cite{MT}. The Kodaira dimension is indeed monotone in dimensions two (obviously), three~\cite{Zhang} (see also~\cite{Neoorder} for an alternative proof based on~\cite{Wang1,KN}), and four~\cite{Neoorder}.

We prove that the Kodaira dimension for geometric 5-manifolds is monotone with respect to Gromov's order.

\begin{thm}\label{t:monotonicity}
Let $M$ and $N$ be two closed oriented geometric 5-manifolds. If $M\geq N$, then $\kappa^g(M)\geq\kappa^g(N)$.
\end{thm}

Before proceeding to our argument, let us first recall some tools and properties that we will need at various stages of the proof.

\subsubsection*{Passing to finite coverings}

We will use virtual properties of manifolds under consideration, such as a desired product or fiber bundle structure. We will do that after lifting our maps as follows: Given a (hypothetical) map of non-zero degree $f\colon M\longrightarrow N$, the group $f_*(\pi_1(M))$ has finite index in $\pi_1(N)$, and so we can lift $f$ to a $\pi_1$-surjective map $\widetilde f\colon M\longrightarrow \widetilde N$, where $\widetilde N\to N$ is the covering corresponding to $f_*(\pi_1(M))$. Sometimes this alone is enough. If we want to achieve further virtual properties, then we consider the finite covering $p\colon\widehat N\to\widetilde N$ (which is also a finite covering of $N$) that has the desired virtual property (e.g., $\widehat N$ has a product or fiber bundle structure). Then there is a covering $q\colon\widetilde M\to M$ corresponding to $\widetilde f_*^{-1}(p_*(\pi_1(\widehat N)))$ such that $\widetilde f\circ q$ lifts to a $\pi_1$-surjective map $\widehat f\colon\widetilde M\longrightarrow\widehat N$. If $\widetilde M$ has the desired properties  (e.g.,  product or fiber bundle structure), then we work with  that map. Otherwise, let $\widehat q\colon\widehat M\to \widetilde M$ be the finite covering with the desired properties and either we work with the map $\widehat f\circ\widehat q\colon\widehat M\longrightarrow\widehat N$ or we lift further $\widehat f\circ\widehat q$ to a $\pi_1$-surjective map.

\subsubsection*{Killing normal subgroups}

In certain cases, after passing to finite covers as explained above, the existence of a normal solvable subgroup in the fundamental group of the domain will simplify the argument. For instance, let $f\colon M\longrightarrow N$ be a $\pi_1$-surjective map, where $M$ and $N$ are aspherical $n$-manifolds. Moreover, suppose $\pi_1(M)$ has non-trivial center $C(\pi_1(M))$, such that $\pi_1(M)/C(\pi_1(M))$ has cohomological dimension $<n$. By $\pi_1$-surjectivity, we obtain $f_*(C(\pi_1(M)))\subseteq C(\pi_1(N))$. Thus, if $C(\pi_1(N))=1$, then we immediately obtain that $H_n(f)=0$, because $f$ factors, up to homotopy, through a space of lower cohomological dimension. Hence, $\deg(f)=0$.

\subsubsection*{Realisation of homology classes by manifolds}

Another tool in showing non-existence of certain maps of non-zero degree is given by Thom's solution~\cite{Thom} of Steenrod's realisation problem~\cite{Eilen}: If $X$ is a topological space and $\alpha\in H_k(X;\Z)$, then there is an integer $d>0$ and a closed $k$-manifold $E$, together with a continuous map $g\colon E\longrightarrow X$, such that $H_k(g)([E])=d\alpha$. For $k\leq 6$, we can take $d=1$. Suppose now $f\colon M\longrightarrow N\times B$ is a map of non-zero degree, and let $\pi\colon N\times B \to B$ be the projection to $B$, where $\dim B<\dim M$. Then, by Poincar\'e Duality, there is a non-trivial homology class $\alpha\in H_{\dim B}(M;\Q)$ such that $H_{\dim B}(\pi\circ f)(\alpha)=[B]$. Thom's theorem~\cite{Thom} guarantees the existence of a manifold $E$ of dimension $\dim B$ together with a continuous map $
h=\pi\circ f\circ g\colon E\longrightarrow B$, such that $H_{\dim B}(h)([E])\neq0\in H_{\dim B}(B;\Z)$. In particular, $E\geq B$. Hence, if we knew that the latter is not possible, i.e., $E\ngeq B$, then we arrive at a contradiction, and so $\deg(f)=0$.

\begin{proof}[Proof of Theorem \ref{t:monotonicity}]
We will show that $M\ngeq N$, whenever $\kappa^g(M)<\kappa^g(N)$. We organise the proof according to the Kodaira dimension of $M$ or $N$. More specifically, we first examine the cases where $\kappa^g(M)=-\infty,0,1$ or $\frac{3}{2}$ and $\kappa^g(N)\neq\frac{5}{2}$. Then we give a uniform treatment for the case  $\kappa^g(N)=\frac{5}{2}$, using only the simplicial volume (although other of our arguments would apply as well), proving in particular Theorem \ref{t:topKodGromov}.

\subsubsection*{Case I: $\kappa^g(M)=-\infty$}

Let $B\pi_1(M)$ be the classifying space of $\pi_1(M)$ and denote by $c_M\colon M\longrightarrow B\pi_1(M)$ the classifying map. Since $M$ is modeled on a geometry which is a (possibly trivial) fibration with a compact fiber or base, we conclude that the induced homomorphism $H_5(c_M)\colon H_5(M;\Q)\longrightarrow H_5(B\pi_1(M);\Q)$ is zero.  On the other hand, if $N$ is a manifold modeled on one of the other geometries with Kodaira dimension $0,1,\frac{1}{2}$ or $2$, then $N$ is aspherical and, thus, its classifying map is homotopic to the identity. Suppose now $f\colon M\longrightarrow N$ is a continuous map and let $Bf_*\colon B\pi_1(M)\longrightarrow B\pi_1(N)$ be the induced map between the classifying spaces. Then there is a commutative diagram as follows.
$$
\xymatrix{
H_5(M;\Q) \ar[r]^{H_5(f)} \ar[d]_{H_5(c_M)} & H_5(N;\Q) \ar[d]^{H_5(c_N)}\\
H_5(B\pi_1(M);\Q) \ \ \ar[r]^{H_5(Bf_*)}  & \ \ H_5(B\pi_1(N);\Q) \\
}
$$
Since $H_5(c_M)=0$ and $H_5(c_N)=id$, we conclude that $$H_5(f)=H_5(c_N\circ f)=H_5(Bf_*\circ c_M)=0,$$ which implies $\deg(f)=0$.

\subsubsection*{Case II: $\kappa^g(M)=0$} Suppose $M$ possesses a solvable geometry and let $f\colon M\longrightarrow N$ be a $\pi_1$-surjective map, i.e., $f_*(\pi_1(M))=\pi_1(N)$. If $N$ has Kodaira dimension $1,\frac{3}{2}$ or $2$, then $\pi_1(N)$ is not solvable, and thus $\deg(f)=0$ by the following group theoretic lemma whose proof is left to the reader.

\begin{lem}\label{l:solvable}
Let $H_1,H_2$ be two groups and $\varphi\colon H_1\longrightarrow H_2$ be a homomorphism. If $H_1$ is solvable, then $\varphi(H_1)\subseteq H_2$ is a solvable subgroup.
\end{lem}

\subsubsection*{Case III: $\kappa^g(M)=1$}

First, let $M$ be a manifold modeled on one of the geometries $\mathbb{H}^2\times\R^3$, $\R^2\times\widetilde{SL_2}$ or $\R^2\rtimes\widetilde{SL_2}$. Then, up to finite covers, $M$ is a circle bundle over a (semi-)direct product $E$ of the 2-torus with a (possibly punctured) hyperbolic surface; see~\cite[Section 5]{Geng3},~\cite[Prop. 6.23 and Table 6.24]{Geng3} and~\cite[Prop. 6.17 and Tables 6.19 and 6.21]{Geng3} respectively. In particular, $\pi_1(M)$ has non-trivial center. 

\begin{rem}\label{r:productcover}
Note that an aspherical 5-manifold $M$ modeled on a non-solvable product geometry $\mathbb X\times\R^k$, $1\leq k\leq3$, is virtually a product of an $\mathbb X$-manifold with the $k$-torus by arguments similar to those of~\cite{Hillman}. Adapting the argument of~\cite[Theorem 9.3]{Hillman}, given for the 4-dimensional geometries $\mathbb H^2\times\R^2$, $\mathbb H^3\times\R$ and $\widetilde{SL}_2\times\R$, the fundamental group of the 5-manifold $M$ has (up to finite index subgroups) center $C(\pi_1(M))\cong\Z^l\times\Z^k$, where $l$ is the maximum rank of the center of the fundamental group of a manifold $N$ modeled on $\mathbb X$, which is $\mathbb H^4$, $\mathbb H^2(\mathbb C)$, $\mathbb H^2\times\mathbb H^2$, $\mathbb H^3$, $\widetilde{SL_2}$ or $\mathbb H^2$ (i.e., $l=0$ or $1$). Then similarly to~\cite[Theorem 9.3]{Hillman} the projection to the Euclidean factor maps $C(\pi_1(M))$ injectively and $\pi_1(M)$ preserves the foliation of the model space by copies of the Euclidean factor. 
\end{rem}

Every map from $E$ to a 4-manifold $B$, which is modeled on one of the geometries $\mathbb H^4$, $\mathbb H^2(\mathbb C)$ or $\mathbb H^2\times\mathbb H^2$, has degree zero, because $\|B\|>0$ and $\|E\|=0$ by~\cite{Gromov,LS,Buch}. Now, every 5-manifold $N$ of Kodaira dimension two is a circle bundle over a 4-manifold modeled on one of the geometries $\mathbb H^4$, $\mathbb H^2(\mathbb C)$ or $\mathbb H^2\times\mathbb H^2$; see~\cite[Prop. 6.23 and Tables 6.24 and 6.29]{Geng3} for the geometries $\mathbb{H}^2\times\widetilde{SL_2}$, $\widetilde{SL_2}\times_\alpha\widetilde{SL_2}$ and $\mathbb{H}^2\times \mathbb{H}^2\times\R$ and~\cite[Prop. 4.1 and 4.2 and Table 4.3]{Geng3} for the geometries $\mathbb{H}^4\times\R$, $\mathbb{H}^2(\mathbb C)\times\R$ and $\widetilde{U(2,1)/U(2)}$. Hence, the following lemma, which is a straightforward generalisation of~\cite[Lemma 5.1]{Neoorder}, tells us that any map $f\colon M\longrightarrow N$ has degree zero. 

\begin{lem}[\cite{Neoorder}]\label{l:factor}
For $i = 1, 2$, let $S^1\to M_i\to B_i$ be circle bundles over closed oriented aspherical manifolds $B_i$ of the same dimension, so that the center of $\pi_1(M_2)$ remains infinite cyclic in finite covers. If $B_1\ngeq B_2$, then $M_1\ngeq M_2$.
\end{lem}

If $N$ has Kodaira dimension $\frac{3}{2}$, then it is virtually a product $F\times T^2$, where $F$ is a hyperbolic 3-manifold; cf.~\cite[Section 5]{Geng3} or Remark \ref{r:productcover}. Since $\pi_1(M)$ contains $\Z^2$ as a normal subgroup (which is moreover central for the geometries $\mathbb{H}^2\times\R^3$ and $\R^2\times\widetilde{SL_2}$), we deduce that any $\pi_1$-surjective map $f\colon M\longrightarrow N$ factors through a map $\overline f\colon B\longrightarrow F$, where $B$ is a 3-manifold modeled on the geometry $\mathbb{H}^2\times\R$, when $M$ is an $\mathbb{H}^2\times\R^3$-manifold, or on the geometry $\widetilde{SL_2}$, when $M$ is an $\R^2\times\widetilde{SL_2}$- or $\R^2\rtimes\widetilde{SL_2}$-manifold. Hence, $\overline f$ factors through a surface, which implies $\deg(\overline f)=0$. By a statement similar to that of Lemma \ref{l:factor}, we deduce that $\deg(f)=0$.

Let now $M$ be modeled on one of the geometries $Nil^3\times\mathbb H^2$ or $Nil^3\times_\R\widetilde{SL_2}$. In that case, $M$ is virtually a circle bundle over $T^2\times\Sigma_g$, $g\geq2$; cf.~\cite[Prop. 6.23 and Tables 6.24 and 6.29]{Geng3}. Hence, as above, there is no map of non-zero degree from $M$ to any manifold of Kodaira dimension two. Suppose now that the target $N$ has Kodaira dimension $\frac{3}{2}$, i.e., it is virtually a product $F\times T^2$, where $F$ is a hyperbolic 3-manifold, and let $f\colon M\longrightarrow F\times T^2$ be a continuous $\pi_1$-surjective map. Let $\pi\colon F\times T^2\longrightarrow F$ be the projection to the $F$-factor. Since the center of $\pi_1(M)$ is infinite cyclic given by the $S^1$-fiber, the composite map $\pi\circ f\colon M\longrightarrow F$ factors through the bundle projection $p\colon M\longrightarrow T^2\times\Sigma_g$. If $H_3(p)=0$, then $H_3(f)=0$, which means that $\deg(f)=0$, because otherwise $H_3(f)\colon H_3(M)\longrightarrow H_3(F\times T^2;\Q)\neq0$ would be surjective. If $H_3(p)\neq0$ and $\deg(f)\neq0$, then there is an induced map $\overline f\colon T^2\times\Sigma_g\longrightarrow F$, such that $H_3(\overline f)\neq0$. Since 
$$H^3(T^2\times\Sigma_g)\cong(H^2(T^2)\otimes H^1(\Sigma_g))\oplus(H^1(T^2)\otimes H^2(\Sigma_g)),$$ we conclude that there is a map of non-zero degree from $T^3$ or $S^1\times\Sigma_g$ to $F$ (cf.~\cite{Neodegrees,Thom}), which is impossible, as such a map would factor through a surface. Thus $\deg(f)=0$.

Finally, let $M$ be a $Sol^3\times\mathbb H^2$-manifold. We may assume (after passing to a finite cover, if necessary) that $M=E\times\Sigma_g$, where $E$ is a mapping torus of an Anosov diffeomorphism of $T^2$; cf.~\cite[Prop. 6.23 and Table 6.24]{Geng3} and~\cite{Scott}. We first observe that every map from $M$ to a manifold that possesses the geometry $\mathbb H^3\times\R^2$ has degree zero: Indeed, suppose $f\colon E\times\Sigma\longrightarrow F\times T^2$ is a map of non-zero degree, where $F$ is a hyperbolic 3-manifold. Let the composition $\pi\circ f\colon E\times\Sigma\longrightarrow F$, where $\pi\colon F\times T^2\longrightarrow F$ is the projection to the $F$-factor. Since 
$$H^3(E\times\Sigma_g)\cong(H^3(E)\otimes H^0(\Sigma_g))\oplus(H^2(E)\otimes H^1(\Sigma_g))\oplus(H^1(E)\otimes H^2(\Sigma_g)),$$ we conclude by~\cite{Neodegrees,Thom} that there is a map of non-zero degree from $E$ or $S^1\times\Sigma_h$ ($h\geq1$) to $F$, which is a contradiction, as such a map would factor through the circle or a surface respectively. This means that $\deg(f)=0$. Similar arguments apply when the target $N$ is an $\widetilde{SL_2}\times\mathbb H^2$-manifold (see~\cite{Neoorder} and also~\cite{Neodifferent} for a general characterisation regarding such products, as well as~\cite{KN} for projections to the geometry $\widetilde{SL_2}$) or $N$ is modeled on $\mathbb H^2\times\mathbb H^2\times\R$, $\mathbb H^4\times\R$ or $\mathbb H^2(\mathbb C)\times\R$; for the last three geometries note that any non-trivial class in
$$H^4(E\times\Sigma_g)\cong(H^3(E)\otimes H^1(\Sigma_g))\oplus(H^2(E)\otimes H^2(\Sigma_g))$$ is realised either by the product of a $Sol^3$-manifold with the circle or a by the product of the 2-torus with a hyperbolic surface.

We are thus left with the cases where the target $N$ is a virtually a non-trivial circle bundle over a hyperbolic or an $\mathbb H^2\times\mathbb H^2$-manifold. In those cases, we will apply the theory of groups (not) infinite index presentable by products developed in~\cite{NeoIIPP}.

\begin{defn}
A group $H$ is called {\em presentable by products} if there exist two infinite elementwise commuting subgroups $H_1,H_2\subseteq H$, such that the multiplication homomorphism $H_1\times H_2\longrightarrow H$ surjects onto a finite index subgroup of $H$. If both $H_i$ can be chosen with $[H:H_i]=\infty$, then $H$ is called {\em infinite index presentable by products} or {\em IIPP}.
\end{defn}

The property IIPP is a sharp refinement between {\em reducible} groups, i.e., groups that have a finite index subgroup which splits as a direct product of two infinite groups, and groups presentable by products, which were introduced in~\cite{KL}. The following gives a criterion such that the conditions IIPP and reducible are equivalent for central extensions.

\begin{thm}{\normalfont(\cite[Theorem D]{NeoIIPP}).}\label{t:IIPPcriterion}
Let $\Gamma$ be a group with center $C(\Gamma)$ such that the quotient $\Gamma/C(\Gamma)$ is not presentable by products. Then, $\Gamma$ is reducible if and only if it is IIPP.
\end{thm}

A prominent class of groups not presentable by products is given by non-elementary hyperbolic groups~\cite{KL}. Hence Theorem \ref{t:IIPPcriterion} applies to the geometry $\widetilde{U(2,1)/U(2)}$, because if $N$ is an $\widetilde{U(2,1)/U(2)}$-manifold, then, up to finite covers, $N$ has the structure of a non-trivial circle bundle over a closed complex hyperbolic 4-manifold $B$. Clearly $\pi_1(N)$ is not reducible, hence, by Theorem  \ref{t:IIPPcriterion}, it is not IIPP. Thus, every map from a $Sol^3\times\mathbb H^2$-manifold to $N$ has degree zero, by the following theorem.

\begin{thm}{\normalfont(\cite[Theorem B]{NeoIIPP}).}\label{t:IIPPdomination}
Let $S^1\to N\to B$ be a circle bundle over a closed oriented aspherical manifold $B$, so that $\pi_1(N)$ is not IIPP and its center remains infinite cyclic in finite covers. Then $P\ngeq N$, for any non-trivial direct product $P$.
\end{thm}

The same argument applies when the target $N$ is a non-trivial circle bundle over a 4-manifold $B$ that possesses the irreducible $\mathbb H^2\times\mathbb H^2$ geometry, because $\pi_1(B)$ is not presentable by products and $\pi_1(N)$ is irreducible, and thus not IIPP~\cite{NeoIIPP}.

The criterion of Theorem \ref{t:IIPPcriterion} is not anymore valid once we relax the condition on $\Gamma/C(\Gamma)$ being not presentable by products. Such an example is given by the fundamental group of a $Nil^5$-manifold $N$, which is irreducible and IIPP~\cite[Section 8]{NeoIIPP}, and fits into the central extension
\[
1\longrightarrow\Z\longrightarrow\pi_1(N)\longrightarrow\Z^4\longrightarrow1.
\]
As shown in~\cite[Section 8]{NeoIIPP}, $N$ still does not admit maps of non-zero degree from products. We will show below that a similar argument applies to the case of 
$\widetilde{SL_2}\times_\alpha\widetilde{SL_2}$-manifolds.

Let $N$ be modeled on $\widetilde{SL_2}\times_\alpha\widetilde{SL_2}$, such that, after passing to finite covers, it is a non-trivial $S^1$-bundle over the product of two hyperbolic surfaces $\Sigma_{h_1}\times\Sigma_{h_2}$, and $\pi_1(N)$ fits into the central extension
\[
1\longrightarrow\Z\longrightarrow\pi_1(N)\longrightarrow\pi_1(\Sigma_{h_1})\times\pi_1(\Sigma_{h_2})\longrightarrow1.
\]
Since $N$ is not modeled on $\widetilde{SL_2}\times\mathbb H^2$ or $\mathbb H^2\times\mathbb H^2\times\R$, we conclude that $\pi_1(N)$ is irreducible. However, $\pi_1(N)$ is IIPP, and a presentation is given by the multiplication
\[
H_1\times H_2\longrightarrow\pi_1(N),
\]
where $H_i=\langle a_1,b_1,...,a_{h_i},b_{h_i}, z \ | \ [a_1,b_1]\cdots[a_{h_i},b_{h_i}]=z^{t}, \ t\in\Z\setminus\{0\}\rangle$. 

Suppose, now, that there exists a $\pi_1$-surjective map $f\colon X_1\times X_2\longrightarrow N$ of non-zero degree, where $0<\dim(X_i)<5$. We then obtain a short exact sequence
\begin{eqnarray}\label{eq.domination}
 1 \longrightarrow \Gamma_1\cap\Gamma_2 \longrightarrow \Gamma_1 \times \Gamma_2 \stackrel{\varphi}\longrightarrow\pi_1(N) \longrightarrow 1,
\end{eqnarray}
where $\Gamma_i := \mathrm{im}(\pi_1(f\vert_{X_i})) \subset\pi_1(N)$, \ $\Gamma_1\cap\Gamma_2 \subseteq C(\pi_1(N)) = \Z$ and $\varphi$ is the multiplication homomorphism. Moreover, we obtain two non-trivial rational homology classes
\begin{eqnarray}\label{class}
 \alpha_i := H_{\dim X_i}(B\pi_1(f\vert_{X_i})\circ c_{X_i})([X_i]) \neq 0 \in H_{\dim X_i}(B\Gamma_i;\Q),
\end{eqnarray}
where $c_{X_i}$ denote the classifying maps; see~\cite{KL} or~\cite{NeoIIPP}. Since $\pi_1(N)$ is irreducible, $\Gamma_1\cap\Gamma_2$ is isomorphic to $\Z$. Now both $\Z$ and $\pi_1(N)$ are Poincar\'e Duality groups of cohomological dimension one and five respectively, hence $\Gamma_1\times\Gamma_2$ is a Poincar\'e Duality group of cohomological dimension $\mathrm{cd}(\Gamma_1\times\Gamma_2)=6$ and each $\Gamma_i$ is a Poincar\'e Duality group~\cite{Bi,JW}. We need to examine the cases where $\mathrm{cd}(\Gamma_1)=1,2$ or $3$.

If $\mathrm{cd}(\Gamma_1)=1$, then $\Gamma_1=\Z$, and so 
 $B\Gamma_1 \simeq B\Z =
S^1$. The non-vanishing of $\alpha_1 \in H_{\dim X_1}(S^1;\Q)$
implies that $\dim X_1 \leq 1$, that is, $X_1=S^1$. Hence $S^1\times X_2\geq N$, which is impossible by the following Factorization Lemma.

\begin{lem}{\normalfont(\cite[Lemma 4.8]{NeoIIPP}).}\label{l:factorization}
Let $S^1\to N\to B$ be a non-trivial circle bundle over a closed oriented aspherical manifold $B$. Suppose that the Euler class of $N$ is not torsion and that the center of $\pi_1(N)$ remains infinite cyclic in finite covers. Then $X\times S^1\ngeq N$ for any closed oriented manifold $X$.
\end{lem}

If $\mathrm{cd}(\Gamma_1)=2$, then $\Gamma_1$ is a surface group~\cite{Ec}. Since $\Z=\Gamma_1\cap\Gamma_2\subseteq C(\Gamma_1)$, we conclude that $\Gamma_1\cong\Z^2$. Since $\Z=\Gamma_1\cap\Gamma_2\subseteq C(\Gamma_2)$, we deduce that $\rank C(\Gamma_1\times\Gamma_2)\geq3$. But $\Gamma_1\times\Gamma_2$ fits into the short exact sequence (\ref{eq.domination}), where $C(\Gamma_1\cap\Gamma_2)=C(\pi_1(N))=\Z$. This gives us a contradiction~\cite[Lemma 6.23]{NeoIIPP}.

The last case is $\mathrm{cd}(\Gamma_i)=3$. Since $\Gamma_i$ are Poincar\'e Duality groups and $C(\Gamma_i)\neq1$, we deduce that $\Gamma_i$ must be fundamental groups of closed 3-manifolds modeled on $\R^3$, $Nil^3$, $\mathbb H^2\times\R$ or $\widetilde{SL_2}$ by theorems of Bowditch~\cite{Bow} and Thomas~\cite{Tho}. The geometry $\R^3$ is excluded due to the rank of the center as above. Hence $B\Gamma_i$ are realised by closed manifolds. (Note that at least one of them must be modeled on $\mathbb H^2\times\R$ or $\widetilde{SL_2}$, because $\pi_1(N)$ is not nilpotent~\cite{NeoIIPP}.) We have shown that there are two non-trivial homology classes $\alpha_i\in H_{\dim X_i}(B\Gamma_i;\Q)$ such that
\begin{eqnarray*}\label{degree}
H_5(B\varphi)(\alpha_1 \times \alpha_2)=  \deg(f)[N].
\end{eqnarray*}
For one of the $\alpha_i$, say $\alpha_1$, we have by (\ref{class}) a continuous map
\[
B\pi_1(f\vert_{X_1})\circ c_{X_1}\colon X_1\longrightarrow B\Gamma_1,
\]
where in our case $X_1=E$ is a $Sol^3$-manifold and $B\Gamma_1$ is realised by a 3-manifold modeled on $Nil^3$,  $\mathbb H^2\times\R$ or $\widetilde{SL_2}$. Since $\alpha_1\neq0$, the above map is non-trivial in degree three homology. This is a contradiction because, by the growth of first Betti number (see for example~\cite{Scott}), there are no maps of non-zero degree from a $Sol^3$-manifold to any 3-manifold possessing one of the geometries $Nil^3$,  $\mathbb H^2\times\R$ or $\widetilde{SL_2}$. Therefore $\deg(f)=0$ as claimed.

\subsubsection*{Case IV: $\kappa^g(M)=\frac{3}{2}$}

Let $M$ be a manifold modeled on $\mathbb H^3\times\R$. We can assume that $M$ is the product of a hyperbolic 3-manifold $F$ and the 2-torus. In particular, $\Z\subset\Z^2=C(\pi_1(M))$. Suppose $f\colon M\longrightarrow N$ is a $\pi_1$-surjective map, where $N$ is a manifold of Kodaira dimension two. In all cases, we can assume that $N$ is a circle bundle whose base $B$ is modeled on one of the geometries $\mathbb H^4$, $\mathbb H^2(\mathbb C)$ or $\mathbb H^2\times\mathbb H^2$. In particular, $C(\pi_1(N))=\Z$. Hence $f_*$ factors through a surjection $\overline{f_*}\colon\pi_1(M)/\Z\longrightarrow\pi_1(B)$, where $\pi_1(M)/\Z$ is realised by $S^1\times F$. Since $S^1\times F\ngeq B$, Lemma \ref{l:factor} implies that $\deg(f)=0$.

\subsubsection*{Case V: $\kappa^g(N)=\frac{5}{2}$}

In this case, $N$ is modeled on one of the geometries $\mathbb{H}^5$, $SL(3,\R)/SO(3)$  or  $\mathbb H^3\times\mathbb H^2$, which implies $\|N\|>0$. This is a consequence of Gromov's theorems~\cite{Gromov} for hyperbolic 5-manifolds and products of hyperbolic 2- and 3-manifolds, by the inequality
\[
\|E_1\times E_2\|\geq\|E_1\|\|E_2\|,
\]
and a consequence of a theorem of Bucher~\cite{Bucher} for $SL(3,\R)/SO(3)$. On the other hand, any manifold $M$ with Kodaira dimension $-\infty,0,1,\frac{1}{2}$ or $2$ has zero simplicial volume. For if $M$ is modeled on a compact geometry, then the classifying space of $\pi_1(M)$ has virtual dimension less than five (see also Case I) and thus Gromov's Mapping theorem~\cite{Gromov} tells us that $\|M\|=0$. If $M$ is virtually a fiber bundle with amenable fiber, then $\|M\|=0$ again by Gromov~\cite{Gromov}. For any manifold which is virtually a product with a factor that belongs in the above cases (and thus has zero simplicial volume), it has zero simplicial volume by Gromov's inequality
\[
\|E_1\times E_2\|\leq\begin{pmatrix}
\dim(E_1\times E_2)\\
\dim(E_1)
\end{pmatrix}\|E_1\|\|E_2\|.
\]
In the remaining cases, $M$ virtually fibers over a compact geometry and thus $\|M\|=0$ again by the Mapping theorem~\cite{Gromov}. By $\|N\|>0$ and $\|M\|=0$ we conclude that $M\ngeq N$; cf. inequality (\ref{eq.functorial}).

\medskip

The proof is now complete.
\end{proof}

\begin{rem}
Note that Case V proves in particular Theorem \ref{t:topKodGromov}.
\end{rem}

\section{K\"ahler manifolds with non-vanishing simplicial volume}\label{s:Gromovnorm}

In this section, we prove Theorem \ref{t:posKod}, giving, in particular, a complete answer to Question \ref{topKod0norm}(1) for K\"ahler $3$-folds. In fact, we will present a uniform treatment for all dimensions, and, then,  known results in algebraic geometry will imply an affirmative answer for K\"ahler $3$-folds. 
This gives further evidence for the compatibility of our axiomatic Kodaira dimension with other existing Kodaira dimensions for manifolds with non-zero simplicial volume. 

We start with a lemma which shows that the simplicial volume is a birational invariant.

\begin{lem}\label{birgn}
Birationally equivalent smooth projective varieties (resp. bimeromorphic smooth K\"ahler manifolds) have the same simplicial volume.
\end{lem}
\begin{proof}
The Mapping theorem of~\cite{Gromov} implies that if there is a continuous map $f\colon X_1\rightarrow X_2$ such that the induced homomorphism of fundamental groups is an isomorphism, then $\|X_1\|=\|X_2\|$. In particular, it applies when $f$ is a blowup. By the weak factorization theorem~\cite{AKMW}, any bimeromorphic map between complex manifolds can be factored as a composition of blowups and blowdowns at a smooth center, and each intermediate variety is a complex manifold. Moreover, if we start with a birational map between smooth projective varieties, then each intermediate variety is a smooth projective variety. Hence, birationally equivalent smooth projective varieties, resp. bimeromorphic K\"ahler manifolds, have the same simplicial volume.
\end{proof}

We first deal with uniruled manifolds. 

\begin{prop}\label{uniruledgn0}
Any uniruled manifold has vanishing simplicial volume. 
\end{prop}
\begin{proof}
For a uniruled $n$-fold $X$, there is a complex $(n-1)$-fold $Y$ and a dominant and generically finite rational map $f\colon Y\times \CP^1\dashrightarrow X$. Up to blowups, we can choose $Y$ to be smooth. 
 
By Hironaka's resolution of singularities~\cite{Hironaka}, there is a birational morphism $g\colon Z\rightarrow Y\times \CP^1$, obtained as the composition of blowups along smooth centers,  such that $f\circ g$ is a morphism. 

Since $\CP^1\cong S^2$, the product inequality for the simplicial volume implies $\|Y\times \CP^1\|=0$ (or by~\cite{Gromov,Yano} due to the circle action).
By the Mapping theorem of~\cite{Gromov}, we have $||Z\|=\|Y\times \CP^1\|=0$. Finally, since $\|Z\|\geq|\deg(f\circ g)|\|X\|\geq\|X\|$, we conclude that $\|X\|=0$.
\end{proof}

Apparently, any uniruled manifold has holomorphic Kodaira dimension $\kappa^h=-\infty$. The converse is one of the major open problems, often attributed to Mumford, in the classification theory of projective manifolds (see e.g.~\cite{BDPP}).

\begin{con}[Mumford]\label{uniruled}
A smooth projective variety with $\kappa^h=-\infty$ is uniruled.
\end{con}

This is known to be true for projective $3$-folds~\cite{Mo}. In general, it follows from the Abundance conjecture, which says that the Kodaira dimension agrees with the numerical Kodaira dimension~\cite{KM}. 

The key for the vanishing of the simplicial volume for smooth projective varieties with $0\le \kappa^h(M)\le n-1$ is the case of $\kappa^h=0$. We have the following conjecture (see for example~\cite[(4.1.6)]{Kol95}).

\begin{con}[Koll\'ar]\label{kod0ame}
Let $X$ be a smooth and proper variety with $\kappa^h(X)=0$. Then $X$ has a finite \'etale cover $X'$ such that $X'$ is birational to the product of an Abelian variety and of a simply connected variety with $\kappa^h=0$. In particular, $\pi_1(X)$ has a finite index Abelian subgroup. 
\end{con}

In particular, an affirmative solution to Conjecture \ref{kod0ame} would imply that such $X$ has amenable (virtually Abelian) fundamental group and thus $\|X\|=0$. In the following, we show that any smooth projective $n$-fold with non-vanishing simplicial volume must have $\kappa^h=n$, up to the above two well known conjectures.

\begin{thm}\label{Kod>0}
Up to (the second part of) Conjecture \ref{kod0ame}, any smooth $2n$-dimensional complex projective variety $M$ with $\kappa^h(M)\geq 0$ and $\|M\|>0$ has $\kappa^h(M)=n$. 

In particular, assuming Conjecture \ref{uniruled}, 
then any smooth projective variety with non-vanishing simplicial volume 
is of general type.
\end{thm}
\begin{proof}
When $\kappa^h(M)>0$, then we know that $M$ admits an Iitaka fibration. Precisely, $M$ is birationally equivalent to a projective manifold $X$ which admits an algebraic fiber space structure $\phi\colon X\rightarrow Y$ over a normal projective variety $Y$ such that the Kodaira dimension of a very general fiber of $\phi$ has Kodaira dimension zero. By Lemma \ref{birgn}, we conclude 
 $\|M\|=\|X\|$. 

We recall a vanishing result which is a corollary of the Mapping Theorem~\cite{Gromov}: If a closed manifold $X$ can be mapped into a topological space $Y$ whose covering dimension $\dim Y<\dim X$, such that the pullback of every point in $Y$ has an amenable neighborhood in $X$, then $\|X\|=0$. In our situation, $Y$ is a normal variety, which could be blown up to a smooth projective variety $Y'$. Since the blowup map is holomorphic, it is an open map by the Open Mapping theorem in complex analysis. Moreover, we know that the smooth manifold $Y'$ has covering dimension $\dim Y$, by sending open subsets to $Y$ through the surjective birational morphism. 

Hence, in our setting, the problem is reduced to showing that every fiber of an Iitaka fibration has a neighborhood whose fundamental group is amenable.

A general fiber of an Iitaka fibration is a smooth projective variety of Kodaira dimension zero, and therefore, by (the second part of) Conjecture \ref{kod0ame}, it has amenable fundamental group. Thus any regular fiber has a product neighborhood which has the same fundamental group as the fiber, thus amenable.

When $1\le k=\kappa^h(M)<n$, we know $\dim_{\mathbb C} Y=k$. We are in the setting of~\cite[Theorem 2.12]{Kol95}, which we recall below for the convenience of the reader.

\begin{thm}\label{Kol2.12}
Let $X$ and $Y$ be irreducible normal complex spaces and $f\colon X\rightarrow Y$ a morphism. Assume that there is a Zariski open dense set $Y^0\subset Y$ such that $f\colon X^0:=f^{-1}(Y^0)\rightarrow Y^0$ is a topological fiber bundle with connected fiber $X_g$. Let $y\in Y$ be a point such that there is an $x\in f^{-1}(y)$ satisfying $\dim_xf^{-1}(y)=\dim X-\dim Y$. Then \begin{enumerate}
\item there is an open neighborhood $y\in U\subset Y$ such that $\im[\pi_1(X_g)\rightarrow \pi_1(f^{-1}(U))]$ has finite index in $\pi_1(f^{-1}(U))$;
\item if $f$ is proper, then $\im[\pi_1(X_g)\rightarrow \pi_1(f^{-1}(y))]$ has finite index in $\pi_1(f^{-1}(y))$;
\item if $f$ is smooth at $x$, then $\im[\pi_1(X_g)\rightarrow \pi_1(f^{-1}(U))]$  is surjective. 
\end{enumerate}
\end{thm}

In our case, we can apply Theorem \ref{Kol2.12}(1). Since  the fundamental group of a general fiber is amenable, it implies that we can find a neighborhood $U$ of $p\in Y$, such that $\pi_1(f^{-1}(U))$ is amenable (in fact, by Theorem \ref{Kol2.12}(2), we also know that any fiber of an Iitaka fibration is amenable). Hence $\|M\|=\|X\|=0$ by the above mentioned vanishing result of~\cite{Gromov}. 
\end{proof}

In dimension no greater than $3$, both Conjectures \ref{uniruled} and \ref{kod0ame} 
are known to be true, by~\cite{Mo} and~\cite{Kol95} respectively. Hence, we have the following. 

\begin{cor}\label{posKod} 
Let $M$ be a smooth complex projective $n$-fold with non-vanishing simplicial volume. Then $\kappa^h(M)$ cannot be $n-1, n-2$ or $n-3$. If, moreover, $n=3$, then $\kappa^h(M)=3$.
\end{cor}
\begin{proof}
The fundamental group of a smooth projective $n$-fold, $n\le 3$, of Kodaira dimension zero is amenable. The case of $n=1,2$ follows from the classification. By~\cite[(4.17.3)]{Kol95},  the fundamental group of a smooth projective $3$-fold of Kodaira dimension zero has a finite index Abelian subgroup. In particular, it is amenable. Hence, the first part of our corollary follows from the argument of Theorem \ref{Kod>0}.

We still need to show that any smooth complex projective $3$-fold with $\kappa^h(M)=-\infty$ must have $\|M\|=0$. It follows from~\cite{Mo} that any complex projective $3$-fold has $\kappa^h=-\infty$ if and only if it is uniruled. Then Proposition \ref{uniruledgn0} implies $\|M\|=0$. 
\end{proof}

We remark that Theorem \ref{Kod>0} also provides an alternative argument that any smooth K\"ahler surface with non-vanishing simplicial volume is a surface of general type: 
First, by classification of complex surfaces, any K\"ahler surface can be deformed to, in particular it is diffeomorphic to, a projective surface. By Theorem \ref{Kod>0}, any smooth projective surface with non-vanishing simplicial volume and nonnegative Kodaira dimension must have $\kappa^h=2$. On the other hand, any K\"ahler surface with $\kappa^h=-\infty$ is rational or ruled, which has vanishing simplicial volume. 

We can answer Question \ref{topKod0norm}(1) for smooth K\"ahler $3$-folds with the help of $K_X$-MMP and the Abundance conjecture which is established for K\"ahler $3$-folds~\cite{HP, CHP}. 
In fact, by~\cite{BHL, Luni}, we know that for any compact K\"ahler manifold $X$ of complex dimension three, there exists a bimeromorphic K\"ahler manifold $X'$ which is deformation equivalent to a projective manifold. Hence, by Lemma \ref{birgn} and Corollary \ref{posKod}, we have 

\begin{thm}\label{Kah3}
If $X$ is a smooth K\"ahler $3$-fold with non-vanishing simplicial volume, then $\kappa^h(X)=3$. 
\end{thm}

Combining Corollary \ref{posKod} and Theorem \ref{Kah3} we obtain Theorem \ref{t:posKod}. By Lemma \ref{birgn}, Corollary \ref{posKod} works for Moishezon manifolds and Theorem \ref{Kah3} works for complex $3$-folds of Fujiki class $\mathcal C$.  

Moreover, any  smooth K\"ahler $n$-fold with $\kappa^h=n-1$ must have vanishing simplicial volume, since it satisfies 
the above mentioned version of algebraic approximation~\cite{BHL}. This approach might be generalized to higher dimensional K\"ahler manifolds. Although there are Voisin's examples in each even complex dimension $\ge 8$  of compact K\"ahler manifolds all of whose smooth bimeromorphic models are homotopically obstructed to being a projective variety~\cite{Voisin}, these examples are all uniruled. In fact, it is conjectured by Peternell that this phenomenon cannot happen when the Kodaira dimension is non-negative~\cite{Luni}. 

There is a more direct approach.  By running the MMP for a K\"ahler $3$-fold $X$, we obtain a $\mathbb Q$-factorial bimeromorphic model $X_{min}$ of $X$ with at worst (isolated) terminal singularities whose canonical bundle $K_{X_{min}}$ is nef. By the Abundance conjecture, which is known for K\"ahler $3$-folds, there is some positive number $m$ such that $mK_{X_{min}}$ is base-point free and that the linear system $|mK_{X_{min}}|$ defines a fibration with base dimension $\kappa^h(X)$. We can blow up the total space to get a fibered smooth K\"ahler manifold. The argument for Theorem \ref{Kod>0} still applies.

\bibliographystyle{amsplain}

\end{document}